\documentclass[12pt,a4paper,reqno]{amsart}
\usepackage[utf8]{inputenc} 
\usepackage[english]{babel}

\usepackage{xcolor}

\usepackage{mathtools}
\usepackage{amsfonts}
\usepackage{amsmath}
\usepackage{amssymb, amscd}
\usepackage{amsthm}
\usepackage[margin=1in]{geometry}
\usepackage[all]{xy}

\newtheorem{theorem}[subsection]{Theorem}

\newtheorem{definition}[subsection]{Definition}

\newcommand{\comment}[1]{}

\newcommand{\bZ}{\mathbb{Z}} 
\newcommand{\classGroups}{\mathcal{G}}
\newcommand{\alphabet}{\mathcal{A}}

\newcommand\Stn[1]{\mathcal{S}_{nb}(#1)}
\newcommand\Stnp[1]{\mathcal{S}_{nb}^\prime(#1)}

\newtheorem{cor}[subsection]{Corollary}
\newcommand{\rrr}{\mathbb R}
\newcommand{\z}{\mathbb Z}

\newcommand{\id}{\operatorname{id}} 

\newcommand\St[1]{\mathcal{S}(#1)}

\newcommand{\ddx}[1]{\mathcal D(M,#1)}

\newcommand{\com}{\left[ G\wr_n\bZ, G\wr_n\bZ \right]}
\newcommand{\ab}[1]{#1^{'}}
\newcommand{\qabc}[1]{#1^{''}}

\newcommand\wrm[1]{\mathop{\wr}\limits_{#1}}
\newcommand{\orb}[2]{\sharp Orb Int_{#1,#2}}

\newcommand{\Ex}[2]{Ext(\Gamma_{#1}, {#2})}
\newcommand{\In}[2]{Int(\Gamma_{#1},{#2})}

\title{Geometric interpretation of First Betti numbers of smooth functions orbits}
\author{Iryna Kuznietsova, Yuliia Soroka}
\address{Department of Algebra and Topology, Institute of Mathematics of NAS of Ukraine, Tereshchenkivska str. 3, Kyiv, 01024, Ukraine}
\curraddr{}
\email{kuznietsova@imath.kiev.ua, sorokayulya@imath.kiev.ua}


\begin{document}

\begin{abstract}
Let $M$ be a 2-disk or a cylinder, and $f$ be a smooth function on $M$ with constant values at $\partial M$, devoid of critical points in $\partial M$, and exhibiting a property wherein for every critical point $z$ of $f$ there is a local presentation of $f$ near $z$ that is a homogeneous polynomial without multiple factors. We consider $V$ to be either the boundary $\partial M$ (in the case of a 2-disk) or one of its boundary components (in the case of a cylinder) and $\mathcal{S}^{'}(f, V)$ to consist of diffeomorphisms preserving $f$, isotopic to the identity relative to $V$. We establish a correspondence: the first Betti number of the $f$-orbit is shown to be equal to the number of orbits resulting from the action of $\mathcal{S}^{'}(f,V)$ on the internal edges of the Kronrod-Reeb graph associated with $f$.

\end{abstract}

\maketitle
\section{Introduction}
   Consider a group $A$ and an integer $n \in \mathbb{Z}$. We define the semidirect product $A\wr_n\bZ$ of $A^n$ and $\bZ$ under the natural action of $\bZ$ on $A^n$ through cyclic shifts of coordinates. Formally, $A\wr_n\bZ = A^n\rtimes_\varphi \bZ$, where the homomorphism $\varphi\colon\mathbb{Z}\to \text{Aut}(A^n)$ is specified by $\varphi(k) (a_{0},\dots,a_{n-1})= (a_{k \operatorname{ mod }n},\dots, a_{n-1+k \operatorname{ mod } n})$ for all $k\in\mathbb{Z}$ and $(a_{0},\dots,a_{n-1})\in A^n$. This semidirect product $A\wr_n\bZ$ is called the {\it wreath product} of $A$ and $\bZ$.
	
   Note that the groups $A\wrm{1}\bZ$ and $A\times\bZ$ are isomorphic, as well as ${1}\wrm{n}\bZ$ and $\bZ$.	
\begin{definition}\label{def:classG}	
Let $\classGroups$ be a minimal class of groups satisfying the following conditions:
	
		\begin{enumerate}
			\item[\rm 1)] $ 1 \in \classGroups$;
			\item[\rm 2)]  if $A, B \in \classGroups$, then $A \times B \in \classGroups$;
			\item[\rm 3)] if $ A \in \classGroups$ and $n\geq 1$, then $A \wr_{n} \bZ\in \classGroups$.
		\end{enumerate}
	\end{definition}

    In other words, a group $G$ belongs to the class $\classGroups$ if and only if $G$ is derived from the trivial group through a finite sequence of operations involving $\times$ and $\wr_{n}\bZ$.
	It can be observed that any group $G\in\classGroups$ can be expressed as a word in the alphabet $\alphabet=\left\lbrace 1 , \bZ, \left( , \right) , \times, \wr_{2}, \wr_{3}, \wr_{4},\dots\right\rbrace $. This word is referred to as a {\it realization} of the group $G$ in the alphabet $\alphabet$. Notably, the realization of a group is not uniquely determined. 
		
	For instance, the same group can have multiple realizations, exemplified by the following instances:
	
	\begin{equation*} 
		\left(1\wrm{3}\bZ\right)\times\bZ=\bZ\times\left(1\wrm{3}\bZ\right)=\bZ\times\bZ=1\times \bZ\times\bZ.	
	\end{equation*}
	We have shown in \cite{MR4223545} that the number of symbols $\bZ$ in the realization of a group $G\in\classGroups$ in the alphabet $\alphabet$ is uniquely determined by $G$. 
	Namely, we proved the following result.	
		
	\begin{theorem}\label{ZC}\cite{MR4223545}
	Let $Z(G)$ and $[G,G]$ represent the center and the commutator subgroup of $G$, respectively. 
	Consider a group $G\in\classGroups$, an arbitrary realization $\omega$ of $G$ in the alphabet $\alphabet $, and $\beta_1 (\omega)$ be the number of symbols $\bZ$ in the realization $\omega$. 	
	 Then, the following isomorphisms hold:
	$$Z(G) \cong G/ [G,G]\cong \bZ^{\beta_1(\omega)}.$$
	
	In particular, the number $\beta_1(\omega)$ depends only on the group $G$.
	\end{theorem}
	
Groups from the class $\classGroups$ emerge as fundamental groups of orbits of Morse functions on surfaces.

Let $M$ be a compact connected surface and $X$ be a closed subset of $M$. Let also $\ddx{X}$ denote the group of $C^\infty$-diffeomorphisms of $M$ fixed on $X$. There exists a natural right action of the group $\ddx{X}$ on the space of smooth functions $C^\infty(M,P)$ defined by the rule: $(f,h)\mapsto f \circ h$, where $h \in \ddx{X}$, $f \in C^\infty(M,P)$, and $P$ is a real line $\mathbb{R}$ or a circle $S^1$.

The {\it stabilizer} of the function $f$ with respect to the action 
 $$\mathcal{S}(f,X)=\{h\in \ddx{X}\ |\  f\circ \ h =f\}$$
consists of diffeomorphisms from $\ddx{X}$ that preserve $f$. The {\it orbit} of $f$ under this action is defined as $$\mathcal{O}(f,X)=\{f\circ  h\,\, |\,  h\in\ddx{X}\}$$
 representing the set of all possible compositions of $f$ with diffeomorphisms from $\ddx{X}$.
 For simplicity, when $X$ is the empty set, we use the notations $\mathcal{S}(f)$ and $\mathcal{O}(f)$.
 
 Furthermore, we define
 $\mathcal{S}^{'}(f, X) = \mathcal{S}(f,X) \cap \mathcal D_{\id}(M, X)$
 which denotes the subgroup of $\mathcal{S}(f)$ consisting of diffeomorphisms isotopic to the identity relative to $X$, even if such an isotopy does not necessarily preserve $f$.

 Endow the spaces $\ddx{X}$, $C^\infty (M,P)$ with Whitney $C^\infty$-topologies.  Let $\mathcal{O}_f(f,X)$ denote the path component of $f$ in $\mathcal{O}(f,X)$.
 
 	A map $f\in C^{\infty}(M,P)$ will be called {\it Morse} if all its critical points are non-degenerate. 
 	Denote by $\mathcal{M}(M, P)$ the space of all Morse maps $f\colon M\to P$. 
 	A Morse map $f$ is considered {\it generic} if it assigns distinct values to distinct critical points.
 	
 The homotopy types of stabilizers and orbits of Morse functions have been computed in a series of papers authored by Sergiy Maksymenko \cite{Maksymenko:AGAG:2006}, \cite{Maksymenko:UMZ:ENG:2012}, Bohdan Feshchenko and Sergiy Maksymenko \cite{Feshchenko:MFAT:2016},\cite{Feshchenko:Zb:2015}, and Elena Kudryavtseva \cite{Kudryavtseva:SpecMF:VMU:2012}, \cite{Kudryavtseva:MathNotes:2012}, \cite{Kudryavtseva:MatSb:ENG:2013}, \cite{Kudryavtseva:ENG:DAN2016}.
 In particular, E. Kudryavtseva demonstrated that when $M\neq S^2$, for each Morse function $f$, there exists a free action of some finite group $H$ on a $k$-torus $(S^1)^k$, such that $\mathcal{O}_f(f)$ is homotopy equivalent to a $k$-dimensional torus $(S^1)^k/H$. In fact, it was shown in \cite{Maksymenko:AGAG:2006} by S.~Maksymenko that if $f$ is generic, then $H$ is trivial, so $\pi_n\mathcal{O}_f(f)\simeq\pi_n((S^1)^k)$, and the general case of nontrivial $G$ was described in \cite{Kudryavtseva:SpecMF:VMU:2012}, \cite {Kudryavtseva:MathNotes:2012}, \cite{Kudryavtseva:MatSb:ENG:2013}, \cite {Kudryavtseva:ENG:DAN2016} by E. Kudryavtseva. Furthermore, precise algebraic structure of such groups for the case $M\neq S^2, T^2$ was described in \cite{Maksymenko:UMZ:ENG:2012}. 
\begin{definition}\label{clasF}
	Denote by  $\mathcal{F}(M,P)$ the space of smooth functions $f\in C^{\infty}(M,P)$ satisfying the following conditions:
	\begin{enumerate}
		\item
  	The function $f$ takes constant value at $\partial M$ and has no critical point in $\partial M$.
		\item
		For every critical point $z$ of $f$ there is a local presentation $f_z\colon \mathbb{R}^2 \to \mathbb{R}$ of $f$ near $z$ such that $f_z$ is a homogeneous polynomial  $\mathbb{R}^2 \to \mathbb{R}$ without multiple factors.
	\end{enumerate}
\end{definition}

	\begin{definition} Let $f \in \mathcal{F}(M, P)$. A compact submanifold $X \subset M$ will be called $f$-saturated if each of its connected components is either
		\begin{enumerate}
			\item a regular component of a level-set of $f$, or
			\item  a compact subsurface whose boundary consists of regular components of some level-sets of $f$.
		\end{enumerate}
		In particular, when all connected components of $X$ have dimension 2, then $X$ will be called an $f$-saturated subsurface.
		In this case the restriction $f|_X : X \rightarrow P$ satisfies conditions (1),(2) of Definition \ref{clasF}, that is $f|_X \in \mathcal{F}(X, P)$.
	\end{definition}

Notice that we have the following inclusions:
\[\mathcal{M}(M, P) \subset  \mathcal{F}(M, P ) \subset C^\infty (M, P ).\]

The~following theorem relating $\pi_0\mathcal{S}^{'}(f,\partial M)$ with the class $\classGroups$ is a direct consequence of results of \cite{Maksymenko:UMZ:ENG:2012}.
\begin{theorem} \cite{Maksymenko:UMZ:ENG:2012}\label{SI}
Let $M$ be a connected compact oriented surface except 2-sphere and  2-torus and let $f\in \mathcal{F}(M, P) $ be a Morse function. Then $\pi_0\mathcal{S}^{'}(f,\partial M)\in\classGroups$.
\qed
\end{theorem}
Iy was also shown in \cite{Maksymenko:AGAG:2006} that if $M$ is a connencted compact surface with negative Euler characteristics, then $\pi_0\mathcal{S}^{'}(f,X)\simeq\pi_1\mathcal{O}_f(f,X)$.
Notice that by Theorem \ref{ZC} for any $G\in\classGroups $ and its arbitrary presentation $\omega$ in the alphabet $\alphabet $ there is the number $\beta_1(\omega)$ depending only on the group $G$. So we will denote $\beta_1(\omega)$ by $\beta_1(G)$.
Moreover, we have the following Corollary from Theorem \ref{ZC}.
\begin{cor}\label{betti}\cite{MR4223545}
			Let $M$ be a connected compact oriented surface distinct from a 2-sphere and a 2-torus, and let $f \in \mathcal{F}(M,P)$. Also let $G=\pi_1\mathcal{O}(f)$, $\omega$ be any realization of $G$ in the alphabet $\alphabet_{\classGroups}$, and let $\beta_1 (\omega)$ be the number of symbols $\bZ$ in the realization $\omega$. Then the first integral homology group $H_1(\mathcal{O}_f(f), \bZ)$ of the orbit $\mathcal{O}_f(f)$ is a free abelian group of rank $\beta_1(\omega)$:
			$$H_1(\mathcal{O}_f(f), \bZ)\simeq\bZ^{\beta_1(\omega)}.$$
			In particular, $\beta_1(\omega)$ is the first Betti number of the orbit $\mathcal{O}_f(f)$.
	\end{cor}

Our main result (Theorem \ref{main}) shows geometric interpretation of such numbers $\beta_1(G)$.
\section{Geometric interpretation of $\beta_1 (\omega)$}
Let $M$ be a smooth compact not necessarily connected surface, $P$ either the real line $\rrr$ or the circle $S^1$.
Consider the partition of $M$ into connected components of the level-sets of $f\in\mathcal{F}(M,P)$. The set of elements of this partition endowed with the quotient topology will be called {\it Kronrod-Reeb graph} of $f$ and denoted by $\Gamma_{f}$. 
Denote by $p$ the corresponding quotient map from $p\colon M\to\Gamma_f$.
Kronrod-Reeb graphs were independently introduced by G. M. Adelson-Velsky and A. S. Kronrod, and G. Reeb. Since each $f$ takes constant values on the connected components of $\partial M$ and has finite critical points its Kronrod Reeb graph really has the structure of a graph.

Since every $h\in\St{f}$ leaves each level-set of $f$ invariant every homeomorphism $h$ induces homeomorphism $\rho(h)\colon\Gamma_{f}\to\Gamma_{f}$.
So $\rho$ is a homomorphism ${\rho}\colon\St{f}\to Homeo(\Gamma_{f})$.

Let $X\subset M$ be compact surface, $W\subset \partial X$ be the union of some connected components of $\partial X$, and $f \in  \mathcal{F}(X, P)$.

\begin{definition}
	Vertex $v\in \Gamma_f$ will be called {\it external} with respect to $W$ if it is of degree 1 and corresponds either to non-degenerate critical point or to a connected components in $X\setminus W$. 
	
	Otherwise, it will be called internal. In other words, $v\in \Gamma_f$ will be called {\it internal} with respect to $W$ if it is either of degree greater than 1 or of degree 1, but corresponding to a degenerate critical point or to a connected component in $W$.

	Edge of $\Gamma_f$ will be called {\it external} with respect to $W$ if it is incidented to some external with respect to $W$ vertex of $\Gamma_f$. Otherwise, it will be called internal.


Denote by $E(\Gamma_{f|_X})$, $\Ex{f|_X}{W}$ and $\In{f|_X}{W}$ the sets of all edges of the graph $\Gamma_{f|_X}$, all external and all internal edges with respect to $X$ and $W$ correspondingly.

Denote by $\orb{X}{W}$ the number of orbits of the action of $ S^{'}(f|_X,W)$ on internal edges of $\Gamma_{f|_X}$.
\end{definition}

Consider the case when $M$ is a cylinder or a disk. According to Theorem 5.5 \cite{Maksymenko14}, inclusion:
\[S^{'}(f,\partial M) \subset S^{'}(f, V)\] 
is a homotopy equivalence, where $V=\partial M$ if $M=D^2$ or $V=S^1\times 0$ if $M$ is a cylinder. 



\begin{theorem} \label{main}
	Let $M$ be a disk $D^2$ or a cylinder $C=S^1\times [0,1]$ and let $f\in\mathcal{F}(M,P)$. Then $\orb{M}{V}=\beta_1(\pi_0 S^{'}(f,V))$.
	
\end{theorem}	
	\begin{proof}

			We will prove the equality $\orb{M}{V}=\beta_1(\pi_0 S^{'}(f,V))$ by induction on the number of edges of ${\Gamma}_f$.
			  
				\textbf{Step 1.} If the graph $\Gamma_f$ has only one edge then $f$ has either:
						\begin{enumerate}
							\item one non-degenerate critical point if $M$ is a 2-disk or no critical points if it is a cylinder. Then edge is external, so $\orb {M}{V}=0$ and by Theorem 5.6 \cite{Maksymenko14} we get that $ \pi_0 S^{'}(f,V)$ is trivial, so $\beta_1(\pi_0 S^{'}(f,V))=0$,
							\item or degenerate critical point if $M$ is a 2-disk. Then the graph has only one internal edge, so $\orb{M}{V}=1$ and by Theorem 5.6 [Maksymenko,14] we get that $ \pi_0 S^{'}(f,V))=\bZ$, so $\beta_1(\pi_0 S^{'}(f,V))=1$. 				
						\end{enumerate}
					
					Hence, $\orb{M}{V}=\beta_1(\pi_0 S^{'}(f,V))$.

			\textbf{Step $n$.}   Suppose the graph $\Gamma_f$ has $n-1$ edges or less and $\orb{M}{V}=\beta_1(\pi_0 S^{'}(f,V))$. Let us check that for the graph $\Gamma_f$ with $n$ edges it is also true that $\orb{M}{V}=\beta_1(\pi_0 S^{'}(f,V))$.
			
		Let $K$ be the nearest to $V$ critical connected component of some level-set of $f$. Denote 
		
		$R_K$~--- $f$-saturated neighborhood of $K$;
		
		${\bf Z}$~--- the set of connected components of $\overline{M\setminus R_K}$;
		
		${\bf Z}^{fix}$~--- the set of invariant connected components of $\overline{M\setminus R_K}$ with respect to the action of $S^{'}(f,V)$;
		
		${\bf Z}^{reg}={\bf Z}\setminus{\bf Z}^{fix}$.
		
		According to Theorem 5.8 [Maksymenko 14] there are two possible cases. Consider the first case.
			 
		{\bf	Case 1.} Let ${\bf Z}^{reg}=\emptyset$, i.e ${\bf Z}={\bf Z}^{fix}=\{X_0, X_1,\dots, X_a\}$. We will enumerate $X_i$ so that $V\subset X_0$. 
		In particular, $X_0$ is always a cylinder.

			
			Let $\alpha\colon \Gamma_{f|_{\overline{M\setminus R_K}}} \to\Gamma_{f}$ be the inclusion and $\mu\colon E\left( \Gamma_{f|_{\overline{M\setminus R_K}}}\right) \to E(\Gamma_{f})$ be a natural bijection defined by the rule: for all $e \in E\left( \Gamma_{f|_{\overline{M\setminus R_K}}}\right), e^{' } \in E(\Gamma_{f})$ such that $\alpha(e) \subseteq  e^{'}$ we have $\mu (e)=e^{'}$.
			
			Thus, there are the following bijections 
			\begin{gather}
			 	\Ex{f|_{\overline{M\setminus R_K}}}{V \cup \partial R_K}\to  \Ex{f}{V},\\
			\In {f|_{\overline{M\setminus R_K}}}{V \cup \partial R_K}\to  \In{f}{V}. \label{orbIn}
			\end{gather}
			
			Denote by $\Stn{f, V}$ the intersection of $\mathcal{S}(f)$ with the group of diffeomorphisms of $M$ fixed on some neighborhood of $V$. 
			By Corollary 7.2  in \cite{Maksymenko14} the following inclusion is a homotopy equivalence: 
			\begin{equation*}
				\Stn{f, V}\subset \St{f, V}.
			\end{equation*}
			Moreover, in case 1 by Lemma 7.1 and Lemma 7.4 in \cite{Maksymenko14} the inclusion
			\begin{equation*}
				\Stnp{f, R_K\cup V}\subset \Stn{f, V}
			\end{equation*}
			is also a homotopy equivalence.
			
			Then by Theorem 5.8 in \cite{Maksymenko14}
			\begin{equation}\label{piproduct}
				\pi_0 S^{'}(f, V)\simeq\prod_{i=0}^a\pi_0 \mathcal{S}^{'}(f|_{X_i}, \partial X_i).			
			\end{equation}
			
			By Theorem 5.5 \cite{Maksymenko14} we have $\pi_0 \mathcal{S}^{'}(f|_{X_0}, \partial X_0)\simeq\z$, so 
			\begin{equation}\label{piproduct}
				\pi_0 S^{'}(f, V)\simeq\prod_{i=1}^a\pi_0 S^{'}(f|_{X_i}, \partial X_i)\times\z.			
			\end{equation}

			We have that isomorphism (\ref{piproduct}) induces epimorphisms $p_i\colon S^{'}(f, V)\to S^{'}(f|_{X_i},\partial X_i)$ for each $i=0,1,\dots, a$. 
			Denote by $G(f)=S(f)/\Delta (f)$, where $\Delta (f)$ is a subgroup of $S(f)$ preserving each element of $\Gamma_f$ and such that if $z$ is degenerate local extreme, tangent map of $f$ in $z$ is identity map. Then $p_i$ and $\rho$ induce homomorphisms 
			\begin{gather*}
				\hat {p}_i: G(f) \to G(f|_{X_i}),\\
				{\prod_{i=0}^{a} \hat{\rho}_i}: \prod_{i=1}^{a} S^{'}(f|_{X_i}, \partial X_i) \to  \prod_{i=1}^{a} G(f|_{X_i}).
			\end{gather*}
			
			Moreover, there is the following commutative diagram
			\[
			\xymatrix{
				S^{'}(f, V)  \ar[d]^{ pr} \ar[r]& S^{'}(f, R_K\cup V) \ar[d]^{pr}\ar[r]^{\prod_{i=0}^{a} p_i} &\prod_{i=0}^{a} S^{'}(f|_{X_i}, \partial X_i)  \ar[d]_{\prod_{i=0}^{a} \hat{pr}_i}\\
				\pi_0 S^{'}(f, V)  \ar[d]^{\widehat{\rho}} \ar[r]&\pi_0 S^{'}(f, R_K\cup V)  \ar[d]^{\widehat{\rho}}\ar[r]^{\prod_{i=0}^{a} p_i}&\prod_{i=0}^{a} \pi_0 S^{'}(f|_{X_i}, \partial X_i)  \ar[d]_{\prod_{i=0}^{a} \hat{\rho}_i}\\
				G (f) \ar[r]^\id&   G (f)\ar[r]^{\prod_{i=0}^{a} \widehat{p_i}}&\prod_{i=0}^{a} G(f|_{X_i}) .
			}
			\]
			
			 It follows from the diagram that for homeomorphism $h \in G(f)$ and edges $e_1, e_2 \in E(\Gamma_{f})$ the equality $h(e_1)=e_2$ holds  if and only if there exist $i$, where $i=1,\dots a$ and $g \in G(f|_{X_i})$ such that $\mu^{-1}(e_1), \mu^{-1}(e_2) \in E(\Gamma_{f|_{X_i}})$ and $g(\mu^{-1}(e_1))=\mu^{-1}(e_2)$.
				
So orbits of the action of $S^{'}(f, V)$ are in bijection with orbits  of actions of $S^{'}(f, \partial X_i)$. Using (\ref{orbIn})  we have
 $$\In{f|_{\overline{M\setminus R_K}}}{V \cup \partial R_K}= \sqcup_{i=0}^a \In{f|_{X_i}}{\partial X_i} $$ 
 
 and 
 \[\orb{M}{V}=\orb{\overline{M\setminus R_K}}{V \cup \partial R_K} =\sum_{i=0}^a\orb{X_i}{\partial X_i}=\sum_{i=1}^a\orb{X_i}{\partial X_i}+1. \]
Hence it is proven that 
\begin{equation}\label{orb1}
	\orb{M}{V}=\sum_{i=1}^a \orb{X_i}{\partial X_i}+1.
	\end{equation}

By induction assumption we get 
\begin{equation}\label{betaorb}
	\sum_{i=1}^a \beta(\pi_0 S^{'}(f|_{X_i}, \partial X_i))=\sum_{i=1}^a \orb{X_i}{\partial X_i}.
\end{equation}

Since all the groups $\pi_0 S^{'}(f, \partial M)$ and $\pi_0 S^{'}(f|_{X_i}, \partial X_i)$ in (\ref{piproduct}) belong to the class $\classGroups$ and using Theorem 1.8 [KuzSor] we obtain
\begin{equation}\label{bet1}
	\beta(\pi_0 S^{'}(f, V))=\sum_{i=1}^a \beta(\pi_0 S^{'}(f|_{X_i}, \partial X_i))+1.
\end{equation}

Indeed, it follows from (\ref{bet1}), (\ref{orb1}), and (\ref{betaorb}) that $\orb{M}{V}=\beta_1(\pi_0 S^{'}(f,\partial M))$ is true for all $f$ whose graph $\Gamma_f$ has $n$ edges.

The case 1 is proven.

{\bf Case 2.} Let the group $S{'}(f,V)$ non-effectively acts on components ${\bf Z} $.
Consider the subgroup $S^{'}({\bf Z}) =\{ h \in S^{'} (f,V) | h(Z)=Z \mbox{ for every } Z \in {\bf Z} \} $ which is the kernel of non-efficiency of action $S^{'} (f,V) $ on ${\bf Z}$. Then according to Theorem 5.8 [Maksymenko 14] the group $S^{'} (f,V)/S^{'}({\bf Z})\simeq \bZ_m$ acts effectively on ${\bf Z}$ and free on ${\bf Z}^{reg}$. If $m \geq 2$, then the action has precisely either one or two fixed elements. Thus, we have the two possible cases:  ${\bf Z}^{fix}=\{X_0\}$ or  ${\bf Z}^{fix}=\{X_0,X_1\}$. Let ${\bf Z}^{reg}=\{Y{'}_1,\dots, Y{'}_b\}$. Hence, $b$ divides $m$ and $c=b/m$ is the number of orbits of action $\bZ_m$. Since $\bZ_m$ freely acts on ${\bf Z}^{reg}$ we choose one element in each orbit and denote them $Y_1,\dots, Y_c$ being 2-disks.

a)  ${\bf Z}^{fix}=\{X_0\}$, where $X_0\supset V$ is a cylinder. 

Then
\begin{equation}
	\pi_0 S^{'}(f, V)\simeq\prod_{i=1}^c\pi_0 S^{'}(f|_{Y_i}, \partial Y_i) \wr_{m} \bZ.			
\end{equation}

b)  ${\bf Z}^{fix}=\{X_0, X_1\}$, where $X_0\supset V$ is a cylinder. 

Then
\begin{equation}
	\pi_0 S^{'}(f, V)\simeq\prod_{i=1}^c\pi_0 S^{'}(f|_{Y_i}, \partial Y_i) \wr_{m} \bZ \times \pi_0 S^{'}(f|_{X_1}, \partial X_1) .			
\end{equation}

All the groups $\pi_0 S^{'}(f, \partial M)$ and $\pi_0 S^{'}(f|_{Y_i}, \partial Y_i)$, $\pi_0 S^{'}(f|_{X_1}, \partial X_1)$ belong to the class $\classGroups$ and using Theorem 1.8 in \cite{MR4223545} we obtain

\[
	\beta(\pi_0 S^{'}(f, V))=\sum_{i=1}^c \beta(\pi_0 S^{'}(f|_{Y_i}, \partial Y_i))+1, \mbox { in case a)},
\] 

\[ 
 	\beta(\pi_0 S^{'}(f, V))=\sum_{i=1}^c \beta(\pi_0 S^{'}(f|_{Y_i}, \partial Y_i))+ 1+\beta(\pi_0 S^{'}(f|_{X_1}, \partial X_1)), \mbox { in case b)}.
\]
 
 	By induction assumption we get 
\[
 		\sum_{i=1}^c \beta(\pi_0 S^{'}(f|_{Y_i}, \partial Y_i))=\sum_{i=1}^c   \orb{Y_i}{\partial Y_i}, \quad  \beta(\pi_0 S^{'}(f|_{X_1}, \partial X_1))= \orb{X_1}{\partial X_1}.
\]
 	
 	It remains to show that
 	\begin{equation}\label{orb2}
 		\orb{M}{V}=\sum_{i=1}^c \orb{Y_i}{\partial Y_i}+1, \mbox { in case a)},
 	\end{equation}
 	\begin{equation}\label{orb3}
 		\orb{M}{{V}}=\sum_{i=1}^c \orb{Y_i}{\partial Y_i}+1 +\orb{X_1}{\partial X_1}, \mbox { in case b)},
 	\end{equation}

 		Notice that for any external vertex of $\Gamma_{f|_{\sqcup_{i=1}^cY_i \cup X_0}}$ under inclusion $\alpha\colon \Gamma_{f|_{\overline{M\setminus R_K}}} \to\Gamma_{f}$ remains external and vice versa. So every internal edge $e \in \Gamma_{\sqcup_{i=1}^cY_i \cup X_0}$  corresponds to some internal edge of  $\Gamma_{f}$.
 		Let $h \in S^{'}(f,\partial M) / S^{'}({\bf Z})$. For every $i \in  \left\lbrace 1, 2, ..., c\right\rbrace $ there is $k \in \left\lbrace 1, 2, ..., m\right\rbrace $ such that $h(Y_i)=Y_{i+ck}$ and if $e$ corresponds to $ Y_i$, then $h(e) \in orb \ e$. Thus there are no orbits of $S^{'}(f,\partial M) $ acting on $Y{'}_i, i=1,...,b$ other than orbits acting on $Y_i, i=1,...,c$. 
 
 	Then 
 	 	\begin{equation*}
  \orb{M}{V}=\orb{\sqcup_{i=1}^cY_i \cup X_0}{\sqcup_{i=1}^c\partial Y_i \cup \partial X_0}.
   	\end{equation*}
   
 And since $Y_1,\dots, Y_c$ are chosen as single elements from each orbit, so in case a)
   \begin{equation*}
   	\orb{\sqcup_{i=1}^cY_i \cup X_0}{\sqcup_{i=1}^c\partial Y_i \cup \partial X_0} =\sum_{i=1}^c \orb{Y_i}{\partial Y_i}+	\orb{X_0}{\partial X_0}=\sum_{i=1}^c \orb{Y_i}{\partial Y_i}+1.
   \end{equation*}
   
In the same way, we have in case b)   
    	
 	\begin{equation*}
 		\orb{M}{V}=\orb{\sqcup_{i=1}^cY_i \cup X_0\cup X_1}{\sqcup_{i=1}^c \partial Y_i \cup \partial X_0\cup X_1} =\sum_{i=1}^c \orb{Y_i}{\partial Y_i}+1 +\orb{X_1}{\partial X_1}.
 	\end{equation*}
  	Hence the equalities (\ref{orb2}, \ref{orb3}) are proven.
 	
  	The case 2 is proven.

	
\end{proof}

\bibliographystyle{amsalpha} 
\bibliography{a1}

\end{document}